\documentclass[preprint,12pt,3p]{elsarticle}
\usepackage{amsmath,amsthm,amsfonts,amssymb}
\usepackage{fullpage}
\usepackage{subfig}
\usepackage{blkarray}
\usepackage{caption}
\usepackage{float}
\usepackage{algpseudocode}
\usepackage{tikz}
\usepackage{hyperref}
\numberwithin{equation}{section}
\usepackage{xcolor}
\usepackage{colortbl}
\usepackage[normalem]{ulem}
\usepackage{sectsty}
\usepackage{calrsfs}

\subsectionfont{\color{astral}}
\sectionfont{\color{astral}}
\usepackage{colortbl}

\DeclareFontFamily{OT1}{pzc}{}
\DeclareFontShape{OT1}{pzc}{m}{it}{<-> s * [0.900] pzcmi7t}{}

\usepackage{amsbsy}
\usepackage{amsmath}
\usepackage{accents}
\newlength{\dhatheight}

\newenvironment{frcseries}{\fontfamily{frc}\selectfont}{}

\newcommand{\pc}[1]{#1^{{\textcircled{\tiny d}}}}

\definecolor{astral}{RGB}{46,116,181}
\subsectionfont{\color{astral}}
\sectionfont{\color{astral}}
\usepackage{colortbl}

\definecolor{darkslategray}{rgb}{0.18, 0.31, 0.31}
\definecolor{warmblack}{rgb}{0.0, 0.26, 0.26}

\newcommand{\core}[1]{#1^{\tiny{\textcircled{\tiny\#}}}}

\definecolor{darkslategray}{rgb}{0.18, 0.31, 0.31}
\definecolor{warmblack}{rgb}{0.0, 0.26, 0.26}

\usepackage{multirow}
\usepackage{mathtools}
\usepackage{algorithm}
\usepackage{algorithmicx}
\makeatletter
\def\bState{\State\hskip-\aLG@thistlm}
\makeatother

\newtheorem{theorem}{Theorem}[section]
\newtheorem{lemma}[theorem]{Lemma}
\newtheorem{corollary}[theorem]{Corollary}
\theoremstyle{definition}
\usepackage{hyperref}
\newtheorem{definition}{Definition}[section]
\newtheorem{remark}{Remark}[section]

\journal{.....}

\begin{document}

\begin{frontmatter}

\title{On  WD and WDMP generalized inverses in rings}

\vspace{-.4cm}

\author{Amit Kumar$^a$ and Debasisha Mishra$^{*a}$}

 \address{
                          $^a$Department of Mathematics,\\
                        National Institute of Technology Raipur\\ Chhattishgrah, India.
                        \\email: amitdhull513@gmail.com  \\
                        email$^*$: kapamath\symbol{'100}gmail.com. \\}
\vspace{-2cm}

\begin{abstract}
Motivated by the very recent work of Gao, Y.,  Chen, J., Wang, J., Zou, H. [Comm. Algebra, 49(8) (2021) 3241-3254; 
MR4283143],  we  introduce two new generalized inverses named weak Drazin (WD) and weak Drazin Moore-Penrose (WDMP) inverses for elements in rings.
 A few of their 
properties are then provided, and the fact that the proposed generalized inverses coincide with different well-known generalized inverses under certain assumptions is established. Further, we discuss additive properties, reverse-order law and forward-order law for WD and WDMP generalized inverses. Some examples are also provided in support of the theoretical results. 
 \end{abstract}


\end{frontmatter}

\newpage
\section{Introduction and Preliminaries}
Let $R$ be a proper unitary ring with involution whose unity is 1. A ring $R$ is said to be a {\it proper ring} if $a^*a=0\implies a=0$ for all $a\in R$.  An {\it involution $*$} is an  anti-isomorphism that satisfies the conditions: 
$$(a+b)^*=a^*+b^*,~~(ab)^*=b^*a^*, \text{ and } (a^*)^*=a \text{ for all }a, b \in R.$$
\noindent Let $a\in R.$ Then, the {\it commutant} and the {\it double commutant} of $a$  are defined by
$$comm(a) = \{x \in R : ax = xa\},$$
and
$$comm^2(a) = \{x \in R : xy = yx \text{ for all } y \in comm(a)\},$$
respectively. By $N(R)$, we denote the set of all nilpotent elements of $R.$ An element $a$  is said to be {\it Hermitian} if $a^*=a$, and  is called  {\it idempotent}  if $a^2=a.$ An element $a$ is {\it quasinilpotent} if $1 + xa \in R^{-1} \text{ for all } x \in comm(a)$, where $R^{-1}$ denotes the set of all the  standard invertible elements of $R$.
\noindent  An element $a\in R$ is {\it Moore-Penrose invertible} if there exists a unique element $x\in R$ that satisfies the equations:
$$(1.) axa=a, ~(2.) xax=x, ~(3.) (ax)^*=ax, \textnormal{ and } (4.) (xa)^*=xa.$$
Then, $x$ is called as the {\it Moore-Penrose} \cite{Mpenrose} inverse  of $a$, and is denoted as $x=a^{\dagger}$. By $R^{\dagger}$, we denote the set of all Moore-Penrose invertible elements of $R$.
An element $a $ is called  {\it Drazin invertible} \cite{Drazin}  if  there exists a unique element $x\in R$ such that
$xa^{k+1}=a^{k},~~ax=xa,\text{ and } ax^2=x$, for some positive integer $k$. If the Drazin inverse of $a$ exists, then it is denoted by $a^d$. The smallest positive integer $k$ is called the {\it Drazin index}, is denoted by $i(a)$. The set of all Drazin invertible elements of $R$ will be denoted by $R^{d}.$  If $i(a) = 1$, then the Drazin inverse of $a$ is called the {\it group inverse} of $a$, and is denoted by $a^{\#}$. The set of group invertible elements of $R$ will be denoted by $R^{\#}$.  An element $a$ is EP  \cite{epl} if and only if it commutes with its Moore-Penrose inverse, i.e., $aa^{\dagger}=a^{\dagger}a.$ The set of all EP invertible elements of $R$ will be denoted by $R^{EP}.$\par
In 2017, Xu {\it et al}. \cite{Xu} proved that  an element $a$ is  {\it core invertible} if there exists a unique element $x\in R$ satisfying the following conditions: 
$$(ax)^*=ax,~~ax^2=x,~ \text{ and }~ xa^2=a,$$
 is denoted by $\core{a}$.  The set of all core invertible elements of $R$ will be denoted by $\core{R}.$ An element $a$  is {\it pseudo core invertible} \cite{Mdrazin}  if there exists a unique element $x\in R$ such that
$$(ax)^*=ax,~~ax^2=x,~ \text{ and }~xa^{k+1}=a^k,$$
for some positive integer $k$. The least positive integer $k$ for which the above equations hold is called the {\it pseudo core index}, and is denoted by $I(a).$  The pseudo core inverse of $a$ is denoted by $\pc{a}$, and $\pc{R}$ denotes the set of all pseudo core invertible elements of $R$. An element $a \in R$ has generalized {\it Hirano inverse} \cite{her1} if there exists $b \in R$ such that $b = bab,~~ b \in comm^2(a), \text{ and } (a^2-ab) \in R^{qnil}$, where $R^{qnil}$ is the set of all quasinilpotent elements of $R$. Chen and Sheibani \cite{her1} proved the following results for Hirano invertible elements.
\begin{theorem}(Theorem 3.1, \cite{her1})\label{kd1}\\
An element $a \in R$  has Hirano inverse if and only if $a - a^3 \in N(R)$.
\end{theorem} 
\begin{theorem}(Theorem 2.1, \cite{her1})\label{kd2}\\
 If $a\in R$ has Hirano inverse, then $a$ has Drazin inverse.
\end{theorem}
\begin{corollary}(Corollary 2.8, \cite{her2})\label{co2.4}\\
 Let $a\in A$-banach algebra. Then, the followings are equivalent:
 \begin{enumerate}[(i)]
     \item  $a$ has generalized Hirano inverse,
     \item  There exists a unique idempotent element $p \in A$ such that $pa =ap$ and  $a^2-p \in A^{qnil}$.
 \end{enumerate}
\end{corollary}
An element $a $ is said to be  {\it right pseudo core invertible} \cite{Gao}  if there exists a unique element $x\in R$ such that
$axa^{k}=a^k,~~ax^2=x, \text{ and } (ax)^*=ax,$
for some positive integer $k$, is denoted as  $\pc{a}_{r}$.  The least positive integer $k$ for which the above equations hold is called the {\it right pseudo core index}, and is denoted by $I(a).$ The set of all right pseudo core invertible elements of $R$ is denoted by $\pc{R}_r$. 
In 2019, Zhu \cite{Huihui} introduced the DMP inverse for an element  which is recalled next. Let $a \in R^{d} \cap R^{\dagger}$. Then any element $x$ satisfying $xax = x$, $xa = a^{d}a$, and $a^kx = a^ka^{\dagger}$ for
some positive integer $k$, is called  {\it DMP inverse} \cite{Huihui}  of $a$. It is unique, and is denoted by $a^{d,\dagger}.$
The smallest positive integer $k$ is called the {\it DMP index} of $a$. The set of all DMP invertible elements of $R$ is denoted by $R^{d,\dagger}.$

In 2016, Wang and Liu \cite{Wang2} proposed a new generalized inverse for matrices called  generalized Drazin (or G-Drazin) inverse, and is as follows. Let $A \in \mathbb{C}^{n\times n}$. Then, a matrix $X \in \mathbb{C}^{n\times n}$ is called {\it G-Drazin inverse} of $A$ if 
$$AXA = A,~XA^{k+1} = A^{k},  ~\text{and } A^{k+1}X = A^{k},$$
where $k = ind(A)$. It is denoted by $X=A^{GD}$. In general, this inverse is not unique.\par
Recently, in 2020, Hern\'andez {\it et al.} \cite{thome3} introduced another generalized inverse called {\it GDMP inverse}. The definition of a GDMP inverse of a matrix is stated next. Let $A \in\mathbb{ C}^{n\times n}$ and $k = ind(A)$. For each $A^{GD} \in A\{GD\}$, a GDMP inverse of $A$, denoted by $A^{GD\dagger}$, is the $n \times n$ matrix $A^{GD\dagger} = A^{GD}AA^{\dagger}$. This inverse is also not unique. Similarly, in 2021, Hern\'andez {\it et al.} \cite{Hernandez} introduced two new generalized inverses of rectangular complex matrices, namely 1MP and MP1-inverses, and  showed that the binary relations induced for these new generalized inverses are  partial orders. 

 If $a$ and $b$ are a pair of invertible elements, then $ab$ is also invertible and the inverse of the product $ab$ satisfying
 $$(ab)^{-1}=b^{-1}a^{-1},$$
 is called as the {\it reverse-order law}. On the other way, $$(ab)^{-1}=a^{-1}b^{-1}$$ is known as the {\it forward-order law}.
While the reverse-order law do not hold for different generalized inverses,  the forward-order law is not true even for invertible elements. The {\it additive property} of invertible elements $a$ and $b$ is 
$$(a+b)^{-1}=a^{-1}+b^{-1}.$$
Similarly, the {\it absorption law} for invertible elements $a$ and $b$ is
$$a^{-1}(a+b)b^{-1}=a^{-1}+b^{-1}.$$
In 1966, Greville \cite{Greville} first obtained sufficient conditions for which the reverse-order law holds for the Moore-Penrose inverse in matrix form, i.e., $(AB)^{\dagger}=B^{\dagger}A^{\dagger}$. Mosi\'c and Djordjevi\'c \cite{Mosic7} extended the reverse-order law involving the Moore-Penrose inverse in matrix setting to elements in ring. The same problem was also considered by several authors for other generalized inverses. For example, Deng \cite{Deng} studied the reverse-order law for the group inverse on Hilbert space. In 2012, Mosi\'c and Djordjevi\'c \cite{Mosic} extended the reverse-order law for the  group inverse in Hilbert space to ring. In 2017, Zhu {\it et al.} \cite{Chen20} discussed the reverse-order law for the inverse along an element. In 2019, Xu {\it et al.} \cite{last} studied the reverse-order law and the absorption  law for the $(b,c)$-inverses in rings. Jin and Benitez \cite{Jin24} proved the absorption law for  Moore-Penrose inverse, Drazin inverse, group inverse, core inverse and dual core inverse, respectively.  In 2021, Gao {\it et al.} \cite{Gao3} provided the reverse-order law, the forward-order law and the absorption law for the generalized core inverse. In 2017, Zhu and Chen \cite{Zhu25} provided the forward-order law  and additive property for the Drazin inverse in a ring. In 2021, Li {\it et al.} \cite{li} studied the forward-order law for the  core inverse in matrix setting. In 2022,  Kumar {\it et al.} \cite{Kumar} discussed several results on  additive properties,  reverse-order law and forward-order law for GD inverse and GDMP inverse of matrices. Zhu {\it et al.} (\cite{Chen19}, \cite{Chen20}, \cite{Chen22}) and Zou {\it et al.} \cite{Chen14} provided several results on  additive properties, reverse-order law and forward-order law.  Zhu {\it et al.} \cite{Chen19} obtained the following additive property of the Moore-Penrose inverse.
\begin{theorem}(Lemma 2.3, \cite{Chen19})\label{t2.9}\\
 Let $a, b \in R^{\dagger}$
such that $a^*b=ab^*=0$. Then $(a + b)^{\dagger} = a^{\dagger} + b^{\dagger}$.
\end{theorem}
\noindent Very recently, Baksalary {\it et al.} \cite{bak18} established  necessary and sufficient conditions for two orthogonal projectors to be the Moore-Penrose additive.

This article  aims to study reverse-order law,  forward-order law, additive property, and absorption law for two new generalized inverses for elements in rings called weak Drazin (WD) inverse and weak Drazin Moore-Penrose (WDMP) inverse. \par
In this context, the article is organized as follows. First, we define WD and WDMP inverses which are  extensions of GD and GDMP inverses, respectively.  In Section \ref{sec3}, we illustrate some properties of WD and WDMP inverses. Also, we show that if WD and WDMP inverses  exist, then the right pseudo core inverse, Drazin inverse, DMP inverse and Hirano inverse exist. In Section \ref{sec4}, we establish the reverse-order law, the forward-order law and the additive property for
WD and WDMP inverse.
 We also propose a few 
 results assuming the additive property, reverse-order law and forward-order law hold for  WD inverse and WDMP inverse.
\section{WD inverse and WDMP inverse}\label{sec3}
In this section, we propose two new generalized inverses called WD inverse and WDMP inverse for elements in a ring, and discuss some of their properties. We then define a relation between a WDMP inverse and the right pseudo core inverse.  In this aspect,  we first introduce the definition of a weak Drazin inverse.
\begin{definition}\label{am1}
Let $a\in R$. If there exists an element $x\in R$ such that satisfies the following equations:
$$axa=a,~~a^{k+1}x=a^k, \text { and } xa^{k+1}=a^k,$$
then $x$ is called a {\it weak Drazin  inverse (WD inverse)} of $a$. It is denoted by $x=a^{WD}$ and $k=ind(a)$ is the {\it weak Drazin index} of $a$. The set of all weak Drazin invertible elements is denoted by $R^{WD}.$

\end{definition}

The definition of a WDMP inverse is motivated by the definition of a GDMP inverse of a matrix.
\begin{definition}\label{am2}
Let $a\in R^{WD}\cap R^{\dagger}$. If there exists an element $y\in R$ such that satisfies the following equations:
$$yay=y,~~ ay=aa^{\dagger}, \text{ and } ya^k=a^{WD}a^k,$$
then $y$ is called a {\it weak Drazin Moore-Penrose inverse (WDMP inverse)} of $a$. It is denoted by $y=a^{WD\dagger}$ and $k=ind(a)$ is the {\it weak Drazin  index} of $a$. The set of all weak Drazin Moore-Penrose invertible elements is denoted by $R^{WD\dagger}.$
\end{definition}
\begin{remark}
For an element $a\in R^{WD}\cap R^{\dagger}$ to be WDMP invertible, we need WD invertibility of $a.$ But,  WD inverse does not exist for all elements in rings. Hence, WDMP inverse does not exist for all elements in rings. For example, let $R=\mathbb{Z}$ be a ring with conjugate involution and $0,1\neq a \in R$.  Then, WDMP inverse of $a$ does not exist. 
\end{remark}

A consequence of the above definition of a WD inverse is shown next as a corollary.

\begin{corollary}\label{c3.1}
Let $a\in R^{WD}.$ Then $a\in R^{d}.$
\end{corollary}
\begin{proof}
We have $a\in R^{WD}$. So, $aa^{WD}a=a$, $a^{WD}a^{k+1}=a^k$ and $a^{k+1}a^{WD}=a^k$ imply that $a\in a^{k+1}R\cap Ra^{k+1}$. It is well-known that if $a\in a^{k+1}R\cap Ra^{k+1}$, then $a\in R^d.$ Hence, $a$ is Drazin invertible.
\end{proof}
Every idempotent element is weak Drazin invertible, and is proved in the next result.
\begin{theorem}\label{dev}
Let $p\in R$ be an idempotent element. Then, $p$ is  WD invertible. Moreover, $p^{WD}=p.$
\end{theorem}
\begin{proof}
We know that $p^2=p$. Now, $ppp=p^2p=p^2=p$. Similarly, $pp^{k+1}=p^{k}$ and $p^{k+1}p= p$ for every positive integer $k.$ Hence, $p$ is  WD invertible and every idempotent element is a self WD inverse. 
\end{proof}


\begin{remark}
Since $aa^{\dagger}$ is an idempotent element, so $(aa^{\dagger})^{WD}=aa^{\dagger}.$
\end{remark}

\begin{theorem}\label{ad5}
If $a\in R^{WD}$, then there exist two idempotent elements say $p$ and $q$ such that
$$axa=a,~ a^kp=0, \text{ and } qa^k=0.$$
\end{theorem}
\begin{proof}
From Definition \ref{am1}, we have
\begin{align}\label{eq3.1}
    axa=a,
   \end{align}
   \vspace{-1.2cm}
   \begin{align}\label{eq3.2}
    a^{k+1}x=a^k,
   \end{align}
   \vspace{-1.2cm}
   \begin{align}\label{eq3.3}
    xa^{k+1}=a^k.
   \end{align}
 Now, $(ax)^2=axax=ax$ by \eqref{eq3.1}, we have $(ax)^2=ax$, so it is an idempotent element. Similarly, $(xa)^2=xa$ is idempotent. Hence, $p=1-ax$ and $q=1-xa$  are also  idempotent. From \eqref{eq3.2}, we have $a^{k+1}x=a^k$,  which implies  $a^k(1-ax)=0$. So, $a^kp=0$. Similarly, $qa^k=0.$
\end{proof}
We show that $a^{WD}aa^{\dagger}$ is a solution of the following equations.
\begin{theorem}\label{ad1}
Let $a\in R$ with involution $*.$ If $a\in R^{WD\dagger}$, then $y=a^{WD}aa^{\dagger}$ is a solution of these equations:
$$yay=y,~~ ay=aa^{\dagger}, \text{ and } ya^k=a^{WD}a^k.$$
\end{theorem}
\begin{proof} Putting $y=a^{WD}aa^{\dagger}$ in equations $yay=y,~~ ay=aa^{\dagger}, \text{ and } ya^k=a^{WD}a^k$, we get
\begin{align}\label{aj1}
    a^{WD}aa^{\dagger}aa^{WD}aa^{\dagger}
    &=a^{WD}aa^{\dagger}(aa^{WD}a)a^{\dagger}\nonumber\\
    &=a^{WD}aa^{\dagger}aa^{\dagger}\nonumber\\
    &=a^{WD}aa^{\dagger},
\end{align}
\vspace{-0.8cm}
\begin{align}
    aa^{WD}aa^{\dagger}&=(aa^{WD}a)a^{\dagger}=aa^{\dagger},\label{aj2}
    \end{align}
    and
    \begin{align}
    a^{WD}aa^{\dagger}a^k&=a^{WD}(aa^{\dagger}a)a^{k-1}=a^{WD}aa^{k-1}=a^{WD}a^k,\label{aj3}
\end{align}
respectively.
From  \eqref{aj1}, \eqref{aj2} and \eqref{aj3}, we can say that $y=a^{WD}aa^{\dagger}$ is a solution of this system of equations.
\end{proof}
From Corollary \ref{c3.1} and Definition \ref{am2}, we can say if $a$ is  WDMP invertible, then it is DMP invertible. 
\begin{corollary}
Let $a\in R^{WD\dagger}.$ Then, $a\in R^{d\dagger}.$
\end{corollary}

\begin{proof}
If $a\in R^{WD\dagger}$, then $a\in R^{WD}$ and $a\in R^{\dagger}$ by Definition \ref{am2}. From Corollary \ref{c3.1}, $a\in R^d$. Therefore, $a\in R^d\cap R^{\dagger}$, i.e., there exists an element $x\in R$ such that $x=a^daa^{\dagger}.$ Hence, $a\in R^{d\dagger}.$
\end{proof}
Now, we prove some  properties of a WDMP inverse with the help of Definition \ref{am2} and Theorem \ref{ad1}.
\begin{lemma}\label{ad2}
Let $a\in R^{WD\dagger}$ and $y$ be a WDMP inverse of $a$. Then, the following conditions hold:
\begin{enumerate}[(i)]
    \item $aya^k=a^k$, for every positive integer $k.$
    \item $ay$ and $ya$ are both  idempotent elements.
    \item there exists an idempotent element $p$ such that $pa^k=0$.
    \item $y(ay)^k=y,$ for every positive integer $k.$
    \item $a^{k+1}ya=a^{k+1}$, where $k=ind(a)$.
   \item  $ya^{k+1}y=a^ka^{\dagger}$, where $k=ind(a)$.
   \item $a^{\dagger}ay=a^{\dagger}.$
\end{enumerate}
\end{lemma}

\begin{theorem}\label{ad3}
If $a\in R^{WD\dagger}$, then a is  right pseudo 
core invertible. Moreover, WDMP inverse of $a$ is the right pseudo core inverse of $a.$
\end{theorem}

If $a$ is Hermitian, then the following properties hold.

\begin{theorem}
If $a\in R$ is Hermitian and $y$ be a WDMP inverse of $a$, then the following conditions hold:
\begin{enumerate}[(i)]
    \item $a^{\dagger}ay$ is the group inverse of $a.$
    \item $a^2y=a.$
    \item $y^2=a^{WD}a^{\dagger}=ya^{\dagger}.$
    \item $a^{WD}a$ is a WDMP inverse of  $a^{\dagger}a.$
    \item $aa^{WD}(aa^{\dagger})^{k+1}=aa^{\dagger}.$
\end{enumerate}
\end{theorem}
In 2015, Karanasios \cite{epl} provided the following result for EP elements.
\begin{theorem}(Theorem 4.21, \cite{epl})\label{t2.7}\\
An element $a \in R$ is EP if and only if $a \in R^{\#} \cap R^{\dagger}$ and
one of the following equivalent conditions hold:
\begin{enumerate}[(i)]
      \item  $a^na^{\dagger} = a^{\dagger}a^n,$ for some $n \geq 1$,
      \item  $(a^{\#})^na^{\dagger}=a^{\dagger}(a^{\#})^n,$ for some $n \geq 1$,
       \item  $(a^{\dagger})^n = (a^{\#})^n,$ for some $n \geq 1$.
\end{enumerate}
\end{theorem}

Every EP element is  WDMP invertible. This is shown next.

\begin{theorem}
If $a\in R^{EP}$, then  $a$ is  WD and WDMP invertible. Moreover, $a^{\#}$ is a WD inverse and WDMP inverse of $a$. 
\end{theorem}
Now, we recall the notion of {\it annihilators} of an element in a ring. The left annihilator of $a\in R$ is given by
${^{\circ}{a}}=\{x\in R:xa=0\} $ and  the right  annihilator of $a$ is given by $a^{\circ}=\{x\in R:ax=0\}$. The following lemma combines Lemma 2.5 and Lemma 2.6 of \cite{mis}.
\begin{lemma}\label{t2.10} Let $a, b \in R$.
 \begin{enumerate}[(i)]
\item If $aR \subseteq bR$, then ${^{\circ}b} \subseteq {^{\circ}a}$.
\item If $b$ is regular and  ${^{\circ}b} \subseteq {^{\circ}a}$, then $aR \subset bR.$
\item  If $Ra \subseteq Rb$, then ${b^{\circ}} \subseteq {a^{\circ}}$.
\item If $b$ is regular and ${b^{\circ}} \subseteq {a^{\circ}}$, then $Ra \subseteq Rb$.
\end{enumerate}
\end{lemma}
The next result provides a relation between WDMP inverse and annihilators.
\begin{theorem}
Let $a\in R^{WD\dagger}$ and $y$ be a WDMP inverse of $a$. Then, 
\begin{enumerate}[(i)]
\item $Ry=Ra^*$ and $aR=y^*R.$
\item ${y}^{\circ}={(a^*)}^{\circ}$ and ${^{\circ}{a}}={^{\circ}{(y^*)}}$. 
\end{enumerate}
\end{theorem}

Next, we present $ay$ and $ya$ are Hirano invertible.
\begin{theorem}\label{ad6}
Let $a\in R^{WD\dagger}$ and $y$ be a  WDMP inverse of $a$. Then, $ay$ and $ya$  are both  Hirano invertible.
\end{theorem}
We present a theorem for  idempotent elements that are used to present upcoming results.
\begin{theorem}\label{lema}
Let $a,b \in R$ and $x,y\in R$ be two idempotent elements. Then, the following holds:
\begin{itemize}
    \item [(i)] $(1-x)a=b$ if and only if $xb=0$ and ${^{\circ}(x)}\subseteq{^{\circ}(a-b)},$
    \item[(ii)] $a(1-y)=b$ if and only if $by=0$ and $ (y)^{\circ}\subseteq(a-b)^{\circ}.$
\end{itemize}
\end{theorem}
       
As $a^{WD}a$ and $aa^{WD}$  are both idempotent elements, by  Theorem \ref{lema} we obtain the following corollaries.
\begin{corollary}
Let $a\in R^{WD}$. Then, for any $b,c\in R$,
\begin{enumerate}[(i)]
    \item $(1-a^{WD}a)b=c$ if and only if $a^{WD}ac=0$ and  ${^{\circ}(a^{WD}}a)\subset{^{\circ}(b-c)}.$
    \item $b(1-a^{WD}a)=c$ if and only if $ca^{WD}a=0$ and  ${(a^{WD}}a)^{\circ}\subset{(b-c)^{\circ}}.$
\end{enumerate}
\end{corollary}
Similarly, $ay$ and $ya$, (where $y$ is a WDMP inverse of $a$) are both  idempotent elements, hence the following holds.
\begin{corollary}
Let $a\in R^{WD\dagger}$. Then, for any $b,c\in R$,
\begin{enumerate}[(i)]
    \item $(1-a^{WD\dagger}a)b=c$ if and only if $a^{WD\dagger}ac=0$ and  ${^{\circ}(a^{WD\dagger}a)}\subseteq{^{\circ}(b-c)}.$
    \item $b(1-a^{WD\dagger}a)=c$ if and only if $ca^{WD\dagger}a=0$ and  ${(a^{WD\dagger}a)^{\circ}}\subseteq{(b-c)^{\circ}}.$
\end{enumerate}
\end{corollary}
We end this section with the annihilator property of a WDMP inverse.
\begin{corollary}
If $a\in R^{WD\dagger}$ and y be a  WDMP inverse of $a$, then $${(a^{WD}a)^{\circ}}\subseteq{(ya)^{\circ}}\subseteq {(a^{k+1})^{\circ}}.$$ 
\end{corollary}

\section{Reverse-order law, Forward-order law and Additive property}\label{sec4}
In this section, we present the additive property, the reverse-order law and the forward-order law for  WD  inverse and  WDMP inverse, respectively.
We start this section with an  example that shows the additive property does not always hold for WD inverse. 
Next, we establish a result for the additive property involving  WD inverse.
\begin{theorem}\label{t4.1}
Let $a,b\in R^{WD}$. If $ab=ba=0$, $ab^{WD}=b^{WD}a=0$, and $a^{WD}b=ba^{WD}=0$, then $(a+b)^{WD}=a^{WD}+b^{WD}$.
\end{theorem}
The next result discusses the reverse-order law  for  WD inverse.

\begin{theorem}\label{t5.2}
Let $a,b\in R^{WD}$. If $ab=ba$ and $bb^{WD}a^{WD}=a^{WD}bb^{WD}$, then $(ab)^{WD}=b^{WD}a^{WD}.$
\end{theorem}


The next result can be proved the steps as in Theorem \ref{t5.2} following similarly.
\begin{theorem}\label{th5.2}
Let $a,b\in R^{WD}$. If $ab=ba$ and $ba^{WD}a=a^{WD}ab$, then $(ab)^{WD}=b^{WD}a^{WD}.$
\end{theorem}
   The forward-order law involving a WD inverse is presented below.
\begin{theorem}\label{t5.3}
Let $a,b\in R^{WD}$  and $k=max\{ind(a), ind(b)\}.$  If $ab=ba$ and $b^{WD}ba=ab^{WD}b$, then $(ab)^{WD}=a^{WD}b^{WD}.$
\end{theorem}
The following example demonstrates Theorem \ref{t5.2}, Theorem \ref{th5.2} and Theorem \ref{t5.3}.
Next example shows that the given conditions in Theorem \ref{t5.2} and Theorem \ref{th5.2} are sufficient but not necessary.
   The triple reverse-order law involving a WD inverse is presented below.
\begin{theorem}\label{thme5.3}
Let $a,b,c\in R^{WD}$  commute  with each other,  and $k=max\{ind(a), ind(b), ind(c)\}.$  If $cc^{WD}b=bcc^{WD}$ and $c^{WD}b^{WD}a=ac^{WD}b^{WD}$, then $(abc)^{WD}=c^{WD}b^{WD}a^{WD}.$
\end{theorem}
Similarly, the triple forward law can be proved, and is stated below.
\begin{theorem}\label{thme5.3}
Let $a,b,c\in R^{WD}$ be commute  with each others  and $k=max\{ind(a), ind(b), ind(c)\}.$  If $aa^{WD}b=baa^{WD}$ and $a^{WD}b^{WD}c=ca^{WD}b^{WD}$, then $(abc)^{WD}=a^{WD}b^{WD}c^{WD}.$
\end{theorem}
If   $a^{WD}ab^{WD}=a^{WD}$  and $a^{WD}bb^{WD}=b^{WD}$, then $a^{WD}(a+b)b^{WD}=a^{WD}ab^{WD}+a^{WD}bb^{WD}=a^{WD}+b^{WD}$, i.e., $a^{WD}(a+b)b^{WD}=a^{WD}+b^{WD}$. If the absorption law  holds for  WD inverse, i.e., $a^{WD}(a+b)b^{WD}=a^{WD}+b^{WD}$, then a WD inverse satisfies a few conditions mentioned in the next result.

\begin{theorem}\label{t5.4}
Let $a,b\in R^{WD}$  and $k=max\{ind(a), ind(b)\}.$  If  $a^{WD}(a+b)b^{WD}=a^{WD}+b^{WD}$, then $aa^{WD}bb^{WD}=aa^{WD}$, $a^{WD}ab^{WD}b=b^{WD}b$, $a^kbb^{WD}=a^k$ and $a^{WD}ab^k=b^k$.
\end{theorem}


\begin{corollary}
 Let $a,b\in R$  and $k=max\{ind(a), ind(b)\}.$  If $a^{WD}(a+b)b^{WD}=a^{WD}+b^{WD}$, then
 \begin{enumerate}[(i)]
 \item $b^kR\subseteq a^{WD}R$ and $a^kR=a^kbR.$
 \item $ {^{\circ}(a^{WD})}\subseteq {^{\circ}(b^k)}$ and  ${^{\circ}(a^{k})}={^{\circ}(a^kb)}$.
 \item $Rab^k=Rb^k$ and $Ra^k\subseteq Rb^{WD}$.
 \item $(b^{WD})^{\circ}\subseteq (a^k)^{\circ}$ and $(ab^k)^{\circ}=(b^k)^{\circ}$.
  \end{enumerate}
\end{corollary}

\begin{theorem}\label{dev5.4}
Let $a,b\in R^{WD}$  and $k=max\{ind(a), ind(b)\}.$  If $b^{WD}aa^{WD}=a^{WD}$  and $b^{WD}ba^{WD}=b^{WD}$, then  $b^{WD}(a+b)a^{WD}=a^{WD}+b^{WD}$. Furthermore,     $b^kaa^{WD}=b^k$, $b^{WD}ba^k=a^k$,
    $bb^{WD}aa^{WD}=bb^{WD}$ and $b^{WD}ba^{WD}a=a^{WD}a$.

\end{theorem}
The next result is in the direction of Theorem  \ref{t4.1}.
\begin{theorem}\label{t3.16}
Let $a,b\in R^{WD\dagger}$. If $ab=ba=0$, $a^*b=ab^*=0$, $ab^{WD}=b^{WD}a=0$, and $a^{WD}b=ba^{WD}=0$, then $(a+b)^{WD\dagger}=a^{WD\dagger}+b^{WD\dagger}$.
\end{theorem} 
An immediate consequence of the above result is shown next as a corollary.
\begin{corollary}
Let $a,b\in R^{WD\dagger}$ be  Hermitian. If $ab=ba=0$, $ab^{WD}=b^{WD}a=0$, and $a^{WD}b=ba^{WD}=0$, then $(a+b)^{WD\dagger}=a^{WD\dagger}+b^{WD\dagger}$.
\end{corollary}

\begin{theorem}
Let $a,b\in R^{WD\dagger}$. If $ab=ba$, $a^*b=b^*a$, $a^{WD}abb^{\dagger}=bb^{\dagger}a^{WD}a,$ and $bb^{WD}a^{WD}=a^{WD}bb^{WD}$, then $(ab)^{WD\dagger}=b^{WD\dagger}a^{WD\dagger}.$
\end{theorem}
When $a$ and $b$  are both idempotent and Hermitian, then the forward-order law for  WDMP inverse  holds under a few conditions. This is shown next.
\begin{theorem}
Let $a,b\in R^{WD\dagger}$ be both  idempotent and Hermitian,  and\\ $k=max\{ind(a), ind(b)\}.$  If $ab=ba$ and $b^{WD}ba=ab^{WD}b$, then $(ab)^{WD\dagger}=a^{WD\dagger}b^{WD\dagger}.$
\end{theorem}
We end this section with a straightforward derivation. 
\begin{theorem}
Let $a,b\in R^{WD\dagger}$. 
\begin{enumerate}
\item If $(ab)^{WD\dagger}=b^{WD\dagger}a^{WD\dagger}$, then the following conditions hold:
\begin{enumerate}[(i)]
    \item $b(ab)^{WD\dagger}a=bb^{\dagger}a^{WD}a$,
    \item $b(ab)^{WD\dagger}a^{k+1}=bb^{\dagger}a^{k},$
    \item $b^{k+1}(ab)^{WD\dagger}a=b^{k+1}b^{\dagger}a^{WD}a.$
\end{enumerate}
\item If $(ab)^{WD\dagger}=a^{WD\dagger}b^{WD\dagger}$, then the following conditions hold:
\begin{enumerate}[(i)]
    \item $a(ab)^{WD\dagger}b=aa^{\dagger}b^{WD}b,$
    \item $a^{k+1}(ab)^{WD\dagger}b=a^{k+1}a^{\dagger}b^{WD}b,$
    \item $a(ab)^{WD\dagger}b^{k+1}=aa^{\dagger}b^{k}.$
\end{enumerate}
\end{enumerate}
\end{theorem}

\section{Conclusion}
The important findings are summarized as follows:
\begin{itemize}
    \item
The notion of WD inverse and WDMP inverse have been introduced in rings.

\item Some relations among  WD inverse, Drazin inverse, DMP inverse, WDMP inverse,
 right pseudo core inverse and Hirano inverse have been established.

\item Finally, we have presented a few sufficient conditions such that the reverse-order law and forward-order law for  WD and WDMP generalized inverses hold.
\end{itemize}

\section{Acknowledgements}
The first author acknowledges the support of the Council of Scientific and Industrial Research-University Grants Commission, India.  The authors  thank Vaibhav Shekhar  for his helpful suggestions on some parts of this article.


\begin{thebibliography}{10}

\bibitem{bak18}{Baksalary, O.M.,  Sivakumar, K.C.; Trenkler, G.,} {\it On the Moore-Penrose inverse of a sum of matrices}, {Linear Multilinear Algebra,} (2022) DOI: 10.1080/03081087.2021.2021132.

\bibitem{her2}{Chen, H.; Sheibani, M.,} {\it Generalized Hirano inverses in rings}, {Comm. Algebra, 47(7) (2019) 2967-2978.}\bibitem{her1}{Chen, H.; Sheibani, M.,} {\it On Hirano inverses in rings}, {Turkish J. Math., 43(4) (2019) 2049-2057.}


\bibitem{Deng}{Deng, C.Y.}, {\it Reverse order law for group inverses}, J. Math. Anal. Appl., 382 (2011) 663-671.

\bibitem{Drazin}{Drazin, M.P.}, {\it Pseudo-inverses in associative rings and semigroups}, Amer.  Math. Monthly., 65 (1958) 506-514.


\bibitem{Mdrazin}{Gao, Y.; Chen, J.}, {\it Pseudo core inverses in rings with involution}, Comm. Algebra, 46(1) (2018) 38-50.
\bibitem{Gao3}{Gao, Y.; Chen, J.; Wang, L.; Zou, H.}, {\it Absorption laws and reverse order laws for generalized core inverses}, {Comm. Algebra, 49(8) (2021) 3241-3254.}

\bibitem{Greville}{Greville, T.N.E.}, 
{\it Note on the generalized inverse of a matrix product}, {SIAM Rev., 8 (1966) 518-521}.

\bibitem{thome3}{Hernández, M.V.; Lattanzi, M.B.;  Thome, N.,} {\it GDMP-inverses of a matrix and their duals}, { Linear Multilinear Algebra, (2020) DOI: 10.1080/03081087.2020.1857678}.

\bibitem{Hernandez} {Hernández, M.V.; Lattanzi, M.B.;  Thome, N.}, {\it From projectors to 1MP and MP1 generalized inverses and their induced partial orders}, {Rev. R. Acad. Cienc. Exactas Fís. Nat. Ser. A Mat. RACSAM, 115(3) (2021) 1-13.}

\bibitem{Jin24}{Jin, H.; Benitez, J.,} {\it The absorption laws for the generalized inverses in rings}, {Electron. J. Linear Algebra, 30 (2015) 827-842.}

\bibitem{epl}{Karanasios, S.,} {\it EP elements in rings and in semigroups with involution and in C*-algebras},  {Serdica Math. J., 41(1) (2015) 83-116.}


 \bibitem{Mpenrose}{Koliha, J.J.; Djordjević, D.S.;  Cvetković, D.}, {\it Moore-Penrose inverse in rings with involution}, { Linear Algebra Appl., 426(2-3) (2007) 371-381.}

\bibitem{Kumar}{Kumar, A.; Shekhar, V.; Mishra, D.,} {\it W -weighted GDMP inverse for rectangular matrices}, {Electron. J. Linear Algebra, 38 (2022)  632-654.}
\bibitem{li}{Li, T.; Mosi\'c, D.; Chen, J.,} {\it The forward order laws for the core inverse}, {Aequationes Math., 95 (2021) 415–431.}

\bibitem{Mosic}{Mosi\'c, D.; Djordjevi\'c, D.S.}, {\it Reverse order law for the group inverse  in rings}, Appl. Math. Comput., 219 (2012) 2526-2534.

 \bibitem{Mosic7}{Mosi\'c, D.; Djordjevi\'c, D.S.}, {\it Further results on the reverse order law for the Moore–Penrose inverse in rings with involution}, 	
{Appl. Math. Comput.,  218 (2011) 1478–1483.}




\bibitem{mis}{Raki\'c, D.S.; Dinci\'c, N.C.; Djordjev\'c, D.S.,} {\it Group, Moore–Penrose, core and dual core inverse in rings with involution}, {Linear Algebra Appl., 463 (2014) 115–133.}


\bibitem{Wang2}{Wang, H.; Liu, X.,} {\it Partial orders based on core-nilpotent decomposition}, {Linear Algebra Appl., 488 (2016) 235-248.}
\bibitem{Gao}{Wang, L.; Mosi\'c, D.; Gao, Y.}, {\it Right core inverse and the related generalized inverses,} {Comm. Algebra, 47(11) (2019) 4749-4762.}

 \bibitem{Xu}{Xu, S.; Chen, J.; Zhang, X.}, {\it New characterizations for core inverses in rings with involution}, Front. Math. China., 12(1) (2017) 231-246.
 
 
\bibitem{last}{Xu, S.; Chen, J.;  Benítez, J.; Wang, D.,} {\it Centralizer’s applications to the (b,c)-inverses in rings}, {Rev. R. Acad. Cienc. Exactas Fís. Nat. Ser. A Mat. RACSAM, 113(3) (2019) 1739-1746}.

\bibitem{Huihui}{Zhu, H.,} {\it On DMP inverses and m-EP elements in rings}, {Linear Multilinear Algebra},  67  (2019)  756-766. 

\bibitem{Chen19}{Zhu, H.; Chen, J.; Patricio, P.,} {\it The Moore-Penrose inverse of differences and products of projectors in a ring with involution}, {Turk J. Math., 40 (2016) 1316–1324.}


\bibitem{Chen20}{Zhu, H.; Chen, J.; Patricio, P.,} {\it Reverse order law for inverse along an element}, {Linear Multilinear Algebra, 65(1) (2017) 166-177.}

\bibitem{Chen22}{Zhu, H.; Chen, J.; Patrıcio, P.;  Mary, X.,} {\it Centralizer’s applications to the inverse along an element}, {Appl. Math. Comput., 315 (2017) 27-33.}



\bibitem{Zhu25}{Zhu, H.; Chen, J.,} {\it Additive and product properties of Drazin inverses of elements in a ring}, {Bull. Malays. Math. Sci. Soc., 40(1) (2017) 259-278.}

\bibitem{Chen14}{Zou, H.; Chen, J.; Patricio, P.,} {\it Reverse order law for core inverse in rings}, {Mediterr. J. Math., 15(3) (2018) 1-17}.
\end{thebibliography}
\end{document}